\documentclass[reqno]{amsart}

\usepackage{aliascnt}
\usepackage{hyperref}
\usepackage{amsfonts}
\usepackage{mathrsfs}
\usepackage{amsmath}
\usepackage{amsmath,graphics}
\usepackage{amssymb}
\usepackage{indentfirst,latexsym,bm,amsmath,pstricks,amssymb,amsthm,graphicx,fancyhdr,float,color}
\usepackage[all]{xy}
\usepackage{amsthm}
\usepackage{amscd}
\usepackage{hyperref}
\usepackage[all]{hypcap}
\newtheorem{theorem}{Theorem}[section]

\newaliascnt{lemma}{theorem}
\newtheorem{lemma}[lemma]{Lemma}
\aliascntresetthe{lemma}

\newaliascnt{conjecture}{theorem}
\newtheorem{conjecture}[conjecture]{Conjecture}
\aliascntresetthe{conjecture}

\newaliascnt{proposition}{theorem}
\newtheorem{proposition}[proposition]{Proposition}
\aliascntresetthe{proposition}

\newaliascnt{corollary}{theorem}
\newtheorem{corollary}[corollary]{Corollary}
\aliascntresetthe{corollary}

\newaliascnt{problem}{theorem}

\aliascntresetthe{problem}

\theoremstyle{definition}

\newaliascnt{definition}{theorem}
\newtheorem{definition}[definition]{Definition}
\aliascntresetthe{definition}

\newaliascnt{example}{theorem}

\aliascntresetthe{example}

\theoremstyle{remark}

\newaliascnt{remark}{theorem}

\aliascntresetthe{remark}

\newaliascnt{remarks}{theorem}

\aliascntresetthe{remarks}

\numberwithin{equation}{section}

\def\wt{\widetilde}

\def\lra{\longrightarrow}

\def\Aut{\text{\rm{Aut\,}}}
\def\({$($}
\def\){$)$}

\def\Pic{\text{{\rm Pic\,}}}

\def\Alb{{\rm Alb}}
\def\simn{\sim_{\rm num}}

\begin{document}

\title{On the Severi type inequalities for irregular surfaces}

\author{Xin Lu}
\address{Institut f\"ur Mathematik, Universit\"at Mainz, Mainz, Germany, 55099}
\email{luxin001@uni-mainz.de}
\thanks{This work is supported by SFB/Transregio 45 Periods, Moduli Spaces and Arithmetic of Algebraic Varieties of the DFG (Deutsche Forschungsgemeinschaft),
and partially supported by National Key Basic Research Program of China (Grant No. 2013CB834202).}

\author{Kang Zuo}
\address{Institut f\"ur Mathematik, Universit\"at Mainz, Mainz, Germany, 55099}
\email{zuok@uni-mainz.de}





\phantomsection
\begin{abstract}
\addcontentsline{toc}{section}{Abstract}
Let $X$ be a minimal surface of general type and maximal Albanese dimension with irregularity $q\geq 2$.
We show that $K_X^2\geq 4\chi(\mathcal O_X)+4(q-2)$ if $K_X^2<\frac92\chi(\mathcal O_X)$,
and also obtain the characterization of the equality.
As a consequence, we prove a conjecture of Manetti on the geography of irregular surfaces
if $K_X^2\geq 36(q-2)$ or $\chi(\mathcal O_X)\geq 8(q-2)$,
and we also prove a conjecture that surfaces of general type and maximal Albanese dimension
with $K_X^2=4\chi(\mathcal O_X)$ are exactly the resolution of double covers of abelian surfaces
branched over ample divisors with at worst simple singularities.
\end{abstract}

\maketitle

\tableofcontents

\section{Introduction}
Let $X$ be a minimal surface of general type and maximal Albanese dimension with irregularity $q\geq 2$.
The Severi inequality states that
$$K_X^2\geq 4\chi(\mathcal O_X).$$
It was proved by Manetti \cite{manetti-03} using the analysis on the one-forms on $X$ provided that $K_X$ is ample;
and it was completely proved by Pardini \cite{pardini-05} based on Albanese lifting technique and limiting argument.
A natural question arising is when the equality holds.
\begin{conjecture}[{\cite[\S\,0]{manetti-03} \& \cite[\S\,5.2]{lopes-pardini-12}}]\label{conj-1}
Let $X$ be a minimal surface of general type and maximal Albanese dimension.
Then $K_X^2=4\chi(\mathcal O_X)$ if and only if
the canonical model of $X$ is a flat double cover of an abelian surface
branched over an ample divisor with at worst simple singularities.
\end{conjecture}
Here we recall that the branch divisor of a double cover is said to have at worst simple singularities
if the multiplicities of the singularities appearing in the process of the canonical resolution
are at most three (see \autoref{def-simple-sing}).
The above conjecture was confirmed by Manetti in \cite{manetti-03} under the assumption that $K_X$ is ample.

The next task in the study of the geography on irregular surfaces of general type and maximal Albanese dimension
is how the irregularity $q$ influences the relation between $K_X^2$ and $\chi(\mathcal O_X)$.
As a first attempt in this direction, Mendes Lopes and Pardini proved in \cite{lopes-pardini-11} that
$$K_X^2 \geq 4\chi(\mathcal O_X)+\frac{10}{3}q-8,\qquad\text{if $K_X$ is ample and $q\geq 5$}.$$
In \cite[Conjecture\,7.5]{manetti-03}, Manetti suggested the following inequality
\begin{conjecture}[Manetti]\label{conj-2}
Let $X$ be a minimal surface of general type and maximal Albanese dimension with $q\geq 4$. Then
\item \begin{equation}\label{eqn-conjectue-1}
K_X^2\geq 4\chi(\mathcal O_X)+4(q-3).
\end{equation}
\end{conjecture}

Our main result is the following.
\begin{theorem}\label{thm-main}
Let $X$ be a minimal surface of general type and maximal Albanese dimension.
Then
\begin{enumerate}
\item \begin{equation}\label{eqn-main-1}
K_X^2\geq \min\left\{\frac92\chi(\mathcal O_X),\,4\chi(\mathcal O_X)+4(q-2)\right\}.
\end{equation}
\item Assume that $K_X^2< \frac92\chi(\mathcal O_X)$. Then
\begin{equation}\label{eqn-main-2}
K_X^2=4\chi(\mathcal O_X)+4(q-2),
\end{equation}
if and only if the canonical model of $X$ is a flat double cover of $Y$ whose branch divisor $R$ has at worst simple singularities
and $K_Y\cdot R=4(q-2)$, where
$$
\left\{\begin{aligned}
&Y \text{~is an abelian surface}, &\quad& \text{if $q=2$};\\
&Y \cong E\times C \text{~with $g(E)=1$ and $g(C)=q-1$},&&\text{if $q\geq 3$}.
\end{aligned}\right.
$$
\end{enumerate}
\end{theorem}

Based on the above result, we obtain a positive answer to \autoref{conj-1} and a partial support of \autoref{conj-2}.
\begin{corollary}\label{cor-main-1}
{\rm(1)} \autoref{conj-1} is true.

{\rm(2)} If $X$ is a minimal surface of general type and maximal Albanese dimension with
$K_X^2\geq 36(q-2)$ or $\chi(\mathcal O_X)\geq 8(q-2)$,
then \begin{equation}\label{eqn-cor-main-1}
K_X^2\geq 4\chi(\mathcal O_X)+4(q-2).
\end{equation}
In particular, \eqref{eqn-conjectue-1} holds in this case.
Moreover if $K_X^2> 36(q-2)$ or $\chi(\mathcal O_X)> 8(q-2)$,
then the equality in \eqref{eqn-cor-main-1} can hold
if and only if the canonical model of $X$ is a flat double cover of $Y$
 whose branch divisor $R$ has at worst simple singularities and $K_Y\cdot R=4(q-2)$, where $Y$ is the same as in \autoref{thm-main}(2).
\end{corollary}

The main idea of the proof is as follows.
\begin{enumerate}
\item[Step 1] We apply Pardini's Albanese lifting technique in \cite{pardini-05} verbatim to construct a sequence of fibred surfaces $f_n:\,\wt X_n \to \mathbb P^1$
such that the slope $\lambda_{f_n}$ converges to $K_X^2/\chi(\mathcal O_X)$ when $n$ tends to infinity.
With the help of \cite[Theorem\,3.1]{lu-zuo-15a}, the proof of \autoref{thm-main} is then reduced to the case where $f_n$'s are all double cover fibrations (cf. \cite[Definition\,1.4]{lu-zuo-15a}) for $n$ sufficiently large.

\item[Step 2] We make use of Xiao's linear bound on the order of the automorphism group of
              surfaces of general type (cf. \cite{xiao-94}) to show that the involution on $\wt X_n$ induces a
              well-defined involution $\sigma$ on $X$ provided that $n=p$ is prime and large enough. Moveover, it is shown that the Albanese map of $X$ factors through $X/\langle\sigma\rangle$.

\item[Step 3] We investigate irregular surfaces of general type and maximal Albanese dimension whose Albanese map factors through a double cover, and then apply induction on the degree of the Albanese map to finish the proof of \autoref{thm-main}.
\end{enumerate}

The above three steps are completed in sections \ref{sec-step-I}, \ref{sec-step-II} and \ref{sec-step-III} respectively.
One of the main advantages of Pardini's Albanese lifting technique, compared to
those in \cite{manetti-03,lopes-pardini-11}, is that the canonical divisor of $X$ does not have to be ample.
This allows us to deal with a large class of surfaces.
We remark that the high dimensional Severi inequality has been recently obtained independently by
Barja \cite{barja-15} and Zhang \cite{zhang-14},
where Pardini's Albanese lifting technique is also crucial to both of their proofs.

{\bf Notations and conventions:}
We work over the complex number. The surface $X$ is a minimal smooth surface of general type and maximal Albanese dimension
with irregularity $q\geq2$. The canonical divisor of $X$ is denoted by $K_X$,
and $\chi(\mathcal O_X)$ is referred to be the characteristic of the structure sheaf of $X$.
The Albanese variety of $X$ is denoted by $\Alb(X)$ and $\Alb_X:\,X \to \Alb(X)$ is the Albanese map.
The surface $X$ is said to be of maximal Albanese dimension if the dimension of the image $\Alb_X(X)$ is $2$.

{\bf Acknowledgements:}
When our paper was prepared, Barja kindly informed us that \autoref{conj-1} was also proved in his joint work \cite{barja-pardini-stoppino-15}
with Pardini and Stoppino using different techniques.

\section{Construction of fibred surfaces}\label{sec-step-I}
In the famous paper \cite{pardini-05}, Pardini introduced a clever Albanese lifting technique to construct a sequence of fibred surfaces $f_n:\,\wt X_n \to \mathbb P^1$
such that the slope $\lambda_{f_n}$ converges to $K_X^2/\chi(\mathcal O_X)$ when $n$ tends to infinity.
For readers' convenience, we recall a brief formulation of the construction.

Let $\Alb_X:\,X \to \Alb(X)$ be the Albanese map of $X$.
Let $L$ be a very ample line bundle on $\Alb(X)$ and $H=(\Alb_X)^*(L)$.
For any integer $n\geq 2$, let $\mu_n:\,\Alb(X) \to \Alb(X)$ be the multiplication by $n$ and consider the following Cartesian diagram:
$$\xymatrix{
X_n \ar[rr]^-{\nu_n}\ar[d]_-{a_n}&&X\ar[d]^-{\Alb_X}\\
\Alb(X) \ar[rr]^-{\mu_n}&& \Alb(X)}$$
Then $X_n$ is a smooth surface of maximal Albanese dimension with the following invariants
$$K_{X_n}^2=n^{2q}K_X^2,\qquad \chi(\mathcal O_{X_n})=n^{2q}\chi(\mathcal O_X).$$
According to \cite[\S\,2, Proposition\,3.5]{lange-birkenhake-92}, we have the following numerical equivalence on $\Alb(X)$:
$$\mu_n^*(L) \simn n^2L.$$
It follows that
$$H_n:=a_n^*(L) \simn n^{-2}\nu_n^*(H).$$
Note that $\deg \nu_n=\deg \mu_n = n^{2q}$. Hence
$$H_n^2=n^{2q-4}H^2,\qquad K_{X_n}\cdot H_n=n^{2q-2}K_X\cdot H.$$
Now let $D_1, D_2 \in |H_n|$ be two general smooth curves, $C_1 = D_1 + D_2$ and
$C_2 \in |2H_n|$ be a general curve, such that $C_2$ is smooth and $C_1$ and $C_2$
intersect transversely at $(2H_n)^2$ points. Let $\wt X_n \to X_n$ be the blow up of
the intersection points of $C_1$ and $C_2$, so that we get a (relatively minimal) fibration $f_n:\, \wt X_n \to \mathbb P^1$ of genus $g_n$
with the following invariants:
\begin{eqnarray*}
g_n&=&1+n^{2q-2}K_X\cdot H +2n^{2q-4}H^2,\\[0.1cm]
K_{f_n}^2 &=& n^{2q}K_X^2+8n^{2q-2}K_X\cdot H+12n^{2q-4}H^2,\\[0.1cm]
\chi_{f_n} &=& n^{2q}\chi(\mathcal O_X)+n^{2q-2}K_X\cdot H +2n^{2q-4}H^2.
\end{eqnarray*}
Hence we obtain a sequence of fibred surfaces $f_n:\,\wt X_n \to \mathbb P^1$ with slopes
$$\lambda_{f_n}=\frac{K_{f_n}^2}{\chi_{f_n}}
=\frac{K_X^2+8n^{-2}K_X\cdot H+12n^{-4}H^2}{\chi(\mathcal O_X)+n^{-2}K_X\cdot H +2n^{-4}H^2}
\lra \frac{K_X^2}{\chi(\mathcal O_X)}, \qquad \text{as~}n \lra +\infty.$$

\begin{theorem}\label{thm-step-I}
Let $f_n$'s be the fibrations constructed as above. Then either
\begin{equation}\label{eqn-double-1}
K_X^2 \geq \frac{9}{2}\chi(\mathcal O_X),
\end{equation}
or the fibration $f_n$ is a double cover fibration (cf. \cite[Definition\,1.4]{lu-zuo-15a}) for any sufficiently large $n$.
\end{theorem}
\begin{proof}
Assume that $K_X^2 < \frac{9}{2}\chi(\mathcal O_X)$, i.e., $K_X^2 \leq \frac{9}{2}\chi(\mathcal O_X)-\frac12$.
We have to show that $f_n$ is a double cover fibration if $n$ is sufficiently large.

As $$\lambda_{f_n} \lra \frac{K_X^2}{\chi(\mathcal O_X)}, \qquad \text{as~}n \lra +\infty,$$
there exists $N_1$ such that
\begin{equation}\label{eqn-step-I-1}
\lambda_{f_n} \leq \frac{K_X^2}{\chi(\mathcal O_X)}+\frac{1}{8\chi(\mathcal O_X)}
\leq \frac92-\frac{3}{8\chi(\mathcal O_X)},\qquad \forall~n\geq N_1.
\end{equation}
On the other hand, note that
$$\frac{18(g_n-1)}{4g_n+3}=\frac{18\big(K_X\cdot H+2n^{-2}H^2\big)}{4K_X\cdot H+8n^{-2}H^2+7n^{2-2q}}
\lra \frac92, \qquad \text{as~}n \lra +\infty.$$
Hence there exists $N_2$ such that
$$\frac{18(g_n-1)}{4g_n+3}\geq \frac92-\frac{1}{4\chi(\mathcal O_X)},
\qquad \forall~n\geq N_2.$$
According to \cite[Theorem\,3.1(ii)]{lu-zuo-15a}, it follows that $f_n$ must be a double cover fibration once $n\geq \max\{N_1,N_2\}$.
The proof is complete.
\end{proof}
\section{Factorization of the Albanese map}\label{sec-step-II}
The main purpose of this section is to prove the following theorem on the factorization of the Albanese map.
\begin{theorem}\label{thm-step-II}
Let $X$ be a surface of maximal Albanese dimension with $K_X^2 < \frac{9}{2}\chi(\mathcal O_X)$.
Then there exists an involution $\sigma$ on $X$ such that
the Albanese map $\Alb_X$ factors through a double cover $\pi:\,X \to Y:=X/\langle\sigma\rangle$,
i.e., the following diagram commutes:
$$\xymatrix{X \ar[rr]^-{\pi} \ar@/_5mm/"1,5"_-{\Alb_X}  && Y=X/\langle\sigma\rangle \ar[rr] &&\Alb_X(X)}$$
\end{theorem}

In order to prove \autoref{thm-step-II}, we may assume that $K_X^2 < \frac{9}{2}\chi(\mathcal O_X)$ in this section.
Hence by \autoref{thm-step-I}, the fibration $f_n$ constructed in the last section
is a double cover fibration for any sufficiently large $n$, i.e.,
there exists an involution $\tilde \sigma_n$ on $\wt X_n$ such that $f_n\circ\tilde\sigma_n=f_n$.
We always assume that $n$ is large enough in this section such that $f_n$ is a double cover fibration.
Remark however that there might be more than one involution on $\wt X_n$, in which case we choose and fix one.
Since $X_n$ is the minimal model of $\wt X_n$, $\tilde \sigma_n$ induces an involution $\sigma_n$ on $X_n$.
First we have the following easy lemma.
\begin{lemma}\label{lem-step-II-1}
Let $a_n:\,X_n \to \Alb(X)$ be the map constructed in \autoref{sec-step-I}. Then
$a_n\circ\sigma_n=a_n$.
\end{lemma}
\begin{proof}
Let $x\in X_n$ be a general point, $x'=\sigma_n(x)$, $y=a_n(x)$, and $y'=a_n(x')$.
Then it suffices to prove $y=y'$.
Suppose on the contrary $y\neq y'$.
Let $\Lambda$ be the pencil generated by $C_1=D_1+D_2$ and $C_2$ in \autoref{sec-step-I}.
Since $L$ is very ample on $\Alb(X)$, the complete linear system $|L|$ (hence also $|2L|$)
separates the two points $y$ and $y'$, from which it follows that $|H_n|$ (hence also $|2H_n|$)
separates the two points $x$ and $x'$.
According to the generality of $D_1,D_2$ and $C_2$ in the construction,
one obtains that $\Lambda$ also separates the two points $x$ and $x'$,
which means that $x'\not \in C_x$ and $x\not\in C_{x'}$, where $C_x$ (resp. $C_{x'}$) is the member of $\Lambda$ containing $x$ (resp. $x'$).
Because $f_n\circ\tilde\sigma_n=f_n$, it follows that $\sigma_n(C)=C$ for any member $C\in \Lambda$.
In particular, $\sigma_n(C_x) = C_x$, from which it follows that $x'=\sigma_n(x)\in \sigma_n(C_x) = C_x$.
This gives a contradiction. The proof is complete.
\end{proof}

The next lemma is key to the proof of \autoref{thm-step-II}, which relies highly on Xiao's linear bound
on the order of the automorphism group of surfaces of general type (cf. \cite{xiao-94}).
\begin{lemma}\label{lem-step-II-2}
If $n=p$ is a prime number and sufficiently large, then $\sigma_p$ induces a well-defined involution $\sigma$ on $X$,
i.e., one has the following commutative diagram
$$\xymatrix{
X_p \ar[rr]^-{\nu_p} \ar[d]_-{\sigma_p} && X \ar[d]^-{\sigma}\\
X_p \ar[rr]^-{\nu_p} && X
}$$
\end{lemma}
\begin{proof}
Let $\mu_n:\,\Alb(X) \to \Alb(X)$ be the multiplication by $n$ on the abelian variety $\Alb(X)$ as in \autoref{sec-step-I}. Then $\mu_n$ is an abelian cover with Galois group
$G_n\cong\big(\mathbb Z/n\mathbb Z\big)^{\oplus 2q}$.
By the construction of $X_n$ in \autoref{sec-step-I}, it is a fibre-product.
Hence $\nu_n:\,X_n \to X$ is also an abelian cover with Galois group
$G_n\cong\big(\mathbb Z/n\mathbb Z\big)^{\oplus 2q}$.
If
\begin{equation}\label{eqn-step-II-1}
\sigma_n \circ \tau \circ \sigma_n^{-1} \in G_n, \qquad \forall\,\tau\in G_n,
\end{equation}
then one defines $\sigma(x)=\nu_n\big(\sigma_n(x_n)\big)$, where $x_n\in X_n$ is any point satisfying $\nu_n(x_n)=x$.
By \eqref{eqn-step-II-1}, $\sigma$ is well-defined and the above diagram is obviously commutative.
Hence it suffices to prove \eqref{eqn-step-II-1} when $n=p$ is a prime number and sufficiently large.

Let $\wt G_p$ be the subgroup of the automorphism group $\Aut(X_p)$ of $X_p$ generated by $G_p$ and $\sigma_p$.
Let $p^d$ be the order of any Sylow $p$-subgroup of $\wt G_p$. Note that $K_{X_p}^2=p^{2q}K_X^2$,
and according to \cite[Theorem\,1]{xiao-94}, there exists a universal coefficient $c$ such that
$$p^d\leq |\wt G_p| \leq |\Aut(X_p)| \leq c K_{X_p}^2=p^{2q}\cdot cK_X^2.$$
Hence $d\leq 2q$ if $p>cK_X^2$. On the other hand, it is clear that $d\geq 2q$ due to the fact that $|G_p|=p^{2q}$.
Therefore $G_p$ is a Sylow $p$-subgroup of $\wt G_p$ if $p>cK_X^2$.
Note that \eqref{eqn-step-II-1} is equivalent to
\begin{equation}\label{eqn-step-II-2}
\sigma_n \circ G_n \circ \sigma_n^{-1} = G_n \text{~as subgroups of $\wt G_n$},
\end{equation}

Assume that $\sigma_n \circ G_n \circ \sigma_n^{-1} \neq G_n$, then the number of Sylow $p$-subgroups of $\wt G_p$
is $n_p\geq 2$. On the other hand, according to Sylow's theorems, it is known that $n_p$ divides $|\wt G_p|/p^{2q}$, and
$$n_p \equiv 1 \mod p.$$
It follows that $n_p\geq p+1$, and
$$p^{2q}\cdot cK_X^2 \geq |\Aut(X_p)|\geq |\wt G_p|\geq n_n\cdot p^{2q} \geq p^{2q}(p+1).$$
It is a contradiction if $p>cK_X^2$.
Hence if $p>cK_X^2$, then \eqref{eqn-step-II-2} holds and hence \eqref{eqn-step-II-1} holds too.
The proof is complete.
\end{proof}

\begin{proof}[Proof of \autoref{thm-step-II}]
The existence of an involution follows from \autoref{lem-step-II-2}.
Let $\sigma$ be as in \autoref{lem-step-II-2}.
Then according to \autoref{lem-step-II-1},
$$\Alb_X\circ\sigma\circ\nu_p=\Alb_X\circ\nu_p\circ\sigma_p=\mu_p\circ a_p\circ\sigma_p=\mu_p\circ a_p=\Alb_X\circ \nu_p.$$
Since $\nu_p:\,X_p \to X$ is surjective, it follows that $\Alb_X\circ\sigma=\Alb_X$,
i.e., the Albanese map $\Alb_X$ factors through the double cover $\pi:\,X \to Y:=X/\langle\sigma\rangle$.
\end{proof}

\section{Irregular surfaces coming from double covers}\label{sec-step-III}
Similar to the last section, we will always assume that $K_X^2 < \frac{9}{2}\chi(\mathcal O_X)$ in this section unless other explicit statements.
According to \autoref{thm-step-II}, the Albanese map $\Alb_X$ factors through a double cover
$\pi:\,X \to Y:=X/\langle\sigma\rangle$.
In the section, we deal with the irregular surfaces coming from double covers.

Let $Y' \to Y$ be the resolution of singularities, and $Y_0$ be a minimal model of $Y'$.
The double cover $\pi$ induces a double cover
\begin{equation}\label{eqn-step-III-1}
\pi_0:\,S_0 \lra Y_0,
\end{equation}
and $X$ is nothing but the minimal smooth model of $S_0$.
We want to compute the invariants of $X$ according to the double cover $\pi_0$.
First we have the following lemma
\begin{lemma}\label{lem-step-III-1}
Let $\Alb_{Y'}:\,Y' \to \Alb(Y')$ be the Albanese map of $Y'$.
Then $\Alb(X)\cong \Alb(Y')$.
Hence $Y'$ (and $Y_0$) is also a smooth surface of maximal Albanese dimension with $q(Y')=q$.
\end{lemma}
\begin{proof}
Since $Y'$ is the resolution of singularities of $Y$, by a possible blow up: $\wt X \to X$,
one gets a generically finite map $\wt X \to Y'$ of degree two induced by $\pi$.
Therefore, we obtain a surjective map $\alpha:\,\Alb(X)=\Alb(\wt X) \to \Alb(Y')$.
On the other hand, one has the composition map $Y' \to Y \to \Alb(X)$ whose image is $\Alb_X(X)$.
By the universal property of the Albanese variety, one obtains a surjective map $\beta:\,\Alb(Y') \to \Alb(X)$.
It is easy to check that $\alpha$ and $\beta$ are inverse to each other.
The proof is complete.
\end{proof}

According to the classification of algebraic surfaces, the Kodaira dimension $\kappa(Y_0)=\kappa(Y')\geq 0$.
More precisely, we have the following possibilities for the minimal model $Y_0$:
$$
\left\{\begin{aligned}
&\kappa(Y_0)=0:&&Y_0 \text{~is an abelian surface, in which case $q=2$};\\
&\kappa(Y_0)=1:&&Y_0 \cong E\times C \text{~with $g(E)=1$ and $g(C)=q-1$, in which case $q\geq 3$};\\
&\kappa(Y_0)=2:&&Y_0\text{~is a minimal surface of general type with $q(Y_0)=q$}.
\end{aligned}\right.
$$

\begin{proposition}\label{prop-step-III-i}
If $\kappa(Y_0)=0$ or $1$, then
\begin{equation}\label{eqn-step-III-i-1}
K_{X}^2 \geq 4\chi(\mathcal O_{X})+4(q-2).
\end{equation}
Moreover, the equality in \eqref{eqn-step-III-i-1} holds
if and only if the canonical model of $X$ is isomorphic to $S_0$, the branch divisor $R_0$ of the double cover $\pi_0$
has at worst simple singularities, and $K_{Y_0}\cdot R_0=4(q-2)$.
\end{proposition}
\begin{proof}
Let $\pi_0$ be the induced double cover as in \eqref{eqn-step-III-1} whose cover datum is
$$\mathcal O_{Y_0}(R_0)\equiv \mathcal L_0^{\otimes 2}.$$
The surface $S_0$ may not be smooth.
To get the smooth model, we perform the canonical resolution as follows.
$$\mbox{}
  \xymatrix{
S_{t} \ar[r]^-{\phi_t}\ar[d]_-{\tilde \pi=\pi_t}&
 S_{t-1}\ar[r]^-{\phi_{t-1}}\ar[d]^-{\pi_{t-1}}&
\cdots \ar[r]^-{\phi_2} & S_1\ar[r]^-{\phi_1}\ar[d]_-{\pi_1}
& S_0 \ar[d]^-{\pi_0}\\
Y_{t} \ar[r]^-{\psi_t}& Y_{t-1}\ar[r]^-{\psi_{t-1}}&
\cdots \ar[r]^-{\psi_2} & Y_1\ar[r]^-{\psi_1} & Y_0}
$$
where $S_t$ is smooth and  $\psi_i$'s are successive blowing-ups resolving the singularities of $R_0$;
the map $\phi_i : X_i\to Y_i$ is the double cover determined by
$$\mathcal O_{Y_i}(R_i)\equiv \mathcal L_i^{\otimes 2}$$ with
$$R_i=\psi_i^*(R_{i-1})-2m_{i-1} E_i,\qquad
\mathcal L_i =\psi_i^*(\mathcal L_{i-1})\otimes \mathcal O_{Y_i}(-m_{i-1} E_i),$$
where $E_i$ is the exceptional divisor of $\psi_i$, $d_i$ is the multiplicity of the singular
point $y_i$ in $R_i$ and $m_i=[d_i/2]$, \big(\,[~] stands for the integral part\big).
\begin{definition}\label{def-simple-sing}
If $d_i=2$ or $3$ for any $0\leq i\leq t-1$, then we say that $R_0$ has at worst simple singularities.
In some literature (eg. \cite{xiao-91}), such singularities are also called negligible singularities.
\end{definition}
\noindent
The invariants of $S_t$ is computed by the following formulas (cf. \cite[\S\,V.22]{bhpv-04}):
\begin{eqnarray}
K_{S_t}^2&=&2K_{Y_0}^2+2K_{Y_0}\cdot R_0+\frac12R_0^2-2\sum_{i=0}^{t-1}(m_i-1)^2;\label{eqn-step-III-i-2}\\
\chi(\mathcal O_{S_t})&=&2\chi(\mathcal O_{Y_0})+\frac14K_{Y_0}\cdot R_0+\frac18R_0^2-\sum_{i=0}^{t-1}\frac12m_i(m_i-1).\label{eqn-step-III-i-3}
\end{eqnarray}

If $Y_0 \cong E\times C$ with $g(E)=1$ and $g(C)=q-1$,
then one has $H\cdot R_0\geq 2$ since $X$ is of general type, where $H=\{p\}\times C\subseteq Y_0$
for any $p\in E$. Hence
\begin{equation}\label{eqn-step-III-3}
K_{Y_0}\cdot R_0 =\big(2g(C)-2\big)H\cdot R_0\geq 4(q-2).
\end{equation}
If $Y_0$ is an abelian surface, then it is clear that $K_{Y_0}\cdot R_0=0=4(q-2)$.

Note that $X$ is nothing but the minimal model of $S_t$, which implies that
$K_{X}^2 \geq K_{S_t}^2$ and $\chi(\mathcal O_{X})=\chi(\mathcal O_{S_t})$.
Hence
\begin{eqnarray*}
K_{X}^2- 4\chi(\mathcal O_{X}) &\geq& K_{S_t}^2 -4\chi(\mathcal O_{S_t})\\
&=& 2\big(K_{Y_0}^2 -4\chi(\mathcal O_{Y_0})\big)+K_{Y_0}\cdot R_0+2\sum_{i=0}^{t-1}(m_i-1)\\
&\geq & K_{Y_0}\cdot R_0\geq 4(q-2).
\end{eqnarray*}

To characterize the equality, first it is clear that
the equality holds if the canonical model of $X$ is isomorphic to $S_0$ and the equality in \eqref{eqn-step-III-3} holds.
Conversely, if the equality in \eqref{eqn-step-III-i-1} holds,
then $X=S_t$, the equality in \eqref{eqn-step-III-3} holds, and $m_i=1$ for any $0\leq i\leq t-1$ according to the above arguments.
The later implies that $d_i=2$ or $3$, i.e., $R_0$ has at worst simple singularities,
from which it follows that the inverse image of these exceptional curves
is a union of $(-2)$-curves in $X=S_t$.
This together with the fact that $Y_0$ contains no rational curves
shows that $S_0$ is just the canonical model of $X$.
\end{proof}

\begin{lemma}\label{lem-step-III-ii-1}
Assume that $Y_0$ is of general type.
\begin{enumerate}
\item if $K_{Y_0}^2 \geq 4\chi(\mathcal O_{Y_0})+4\big(q(Y_0)-2\big)$, then
\begin{equation}\label{eqn-step-III-ii-1}
K_{X}^2 \geq 4\chi(\mathcal O_{X})+4(q-2),
\end{equation}
and the the equality can hold only when $q=2$, $\pi=\pi_0$ is unramified and $K_{Y}^2 = 4\chi(\mathcal O_{Y})$;

\item if $K_{Y_0}^2 \geq \frac92\chi(\mathcal O_{Y_0})$, then
\begin{equation}\label{eqn-step-III-ii-1'}
K_{X}^2 > 4\chi(\mathcal O_{X})+4(q-2).
\end{equation}
\end{enumerate}
\end{lemma}
\begin{proof}
We follow the notations introduced in \autoref{prop-step-III-i}.

(1) If $K_{Y_0}^2 \geq 4\chi(\mathcal O_{Y_0})+4\big(q(Y_0)-2\big)$,
then by \eqref{eqn-step-III-i-2} and \eqref{eqn-step-III-i-3},
\begin{eqnarray*}
K_{X}^2- 4\chi(\mathcal O_{X}) &\geq& K_{S_t}^2 -4\chi(\mathcal O_{S_t})\\
&=&2\big(K_{Y_0}^2 -4\chi(\mathcal O_{Y_0})\big)+K_{Y_0}\cdot R_0+2\sum_{i=0}^{t-1}(m_i-1)\\
&\geq & 2\big(K_{Y_0}^2 -4\chi(\mathcal O_{Y_0})\big)\\
&\geq & 8\big(q(Y_0)-2\big)= 8\big(q-2\big) \geq 4(q-2).
\end{eqnarray*}
This proves \eqref{eqn-step-III-ii-1}.
If the equality in \eqref{eqn-step-III-ii-1} holds,
then $q=2$, $K_{Y_0}^2 =4\chi(\mathcal O_{Y_0})$,
$X=S_t$ is minimal, $m_i=1$ for any $0\leq i \leq t-1$ and $K_{Y_0}\cdot R_0=0$.
The last implies that $R_0=\emptyset$ or $R_0$ is a union of $(-2)$-curves since $Y_0$ is minimal of general type.
If $R_0=\emptyset$, then $\pi=\pi_0$ is unramified and hence $Y=Y_0$ as required.
Hence it suffices to deduce a contradiction if $R_0$ is a union of $(-2)$-curves.

Let $D=\sum\limits_{i=1}^{l} C_i \subseteq R_0$ be a connected component.
First we claim that $D$ has at worst simple double points as its singularities;
otherwise, let $p$ be a singularity of $D$ which is not a simple double point,
and $D'=\sum\limits_{i=1}^{k}C_i\subseteq D$ be all the irreducible components of $D$ through $p$.
Then either $k\geq 3$ and $C_i\cdot C_j\geq 1$ for any $1\leq i< j\leq k$, or $k=2$ and $C_1\cdot C_2\geq 2$.
Hence $$(D')^2= \sum\limits_{i=1}^{k}C_i^2 + 2\sum_{1\leq i< j\leq k}C_i\cdot C_j\geq 0.$$
This contradicts the Hodge index theorem, which asserts that for any divisor $L\in \Pic(Y_0)$,
$$(K_{Y_0}\cdot L)^2\geq K_{Y_0}^2\cdot L^2,$$
and the equality holds if and only if $K_{Y_0} \simn r L$ for some rational number $r$.
Hence $D$ has at worst simple double points as its singularities.

Let $\delta$ be the number of its singularities.
If $\delta=0$, i.e., $D=C_1$ is irreducible, then it is clear that the inverse image of $D$ is a $(-1)$-curve,
which contradicts the fact that $X=S_t$ is minimal.
Hence $\delta>0$.
Since $D$ is a connected component of the branch divisor of a double cover,
$(D-C_i)\cdot C_i\geq 2$ is even for any $1\leq i\leq l$.
It follows that $D\cdot C_i=C_i^2+(D-C_i)\cdot C_i\geq 0$,
which implies that $D^2=D\cdot \sum\limits_{i=1}^{l} C_i\geq 0$.
This is also a contradiction to the Hodge index theorem.
The proof is complete.\vspace{0.1cm}

(2)
If $K_{Y_0}^2 \geq \frac92\chi(\mathcal O_{Y_0})$, then by \eqref{eqn-step-III-i-2} and \eqref{eqn-step-III-i-3}
together with the assumption that $K_X^2<\frac92\chi(\mathcal O_{X})$ one obtains
\begin{eqnarray*}
0&>&K_{X}^2-\frac92\chi(\mathcal O_{X})\geq K_{S_t}^2-\frac92\chi(\mathcal O_{S_t})\\
&=& 2\big(K_{Y_0}^2 -4\chi(\mathcal O_{Y_0})\big)
       +\frac78K_{Y_0}\cdot R_0-\frac{1}{16}R_0^2+\frac14\sum_{i=0}^{t-1}(m_i-1)(m_i+8)\\
&\geq& \frac78K_{Y_0}\cdot R_0-\frac{1}{16}R_0^2.
\end{eqnarray*}
Hence $R_0^2> 14K_{Y_0}\cdot R_0$.
On the other hand, according to the Hodge index theorem, one has
$\big(K_{Y_0}\cdot R_0\big)^2 \geq K_{Y_0}^2\cdot R_0^2.$
Thus
$$K_{Y_0}\cdot R_0 > 14 K_{Y_0}^2.$$
Therefore
\begin{eqnarray*}
K_{X}^2- 4\chi(\mathcal O_{X}) &\geq& K_{S_t}^2 -4\chi(\mathcal O_{S_t})\\
&=& 2\big(K_{Y_0}^2 -4\chi(\mathcal O_{Y_0})\big)+K_{Y_0}\cdot R_0+2\sum_{i=0}^{t-1}(m_i-1)\\
&>& \chi(\mathcal O_{Y_0})+14K_{Y_0}^2\geq  64\chi(\mathcal O_{Y_0}).
\end{eqnarray*}
Since $Y_0$ is of general type, one has the inequality $h^0(X,K_X)\geq 2q-4$ (cf. \cite[Corollary\,X.8]{beauville-83}).
In other word, we have
$$\chi(\mathcal O_{Y_0}) \geq \min\big\{1,\,q-3\big\}.$$
Hence
\[K_{X}^2- 4\chi(\mathcal O_{X})> 64\cdot \min\big\{1,\,q-3\big\} > 4(q-2).\qedhere\]
\end{proof}

To deal with the case when $q=2$, we also need the following lemma.
\begin{lemma}\label{lem-step-III-2}
Let $\tilde \zeta:\, \wt W \to Z$ be a flat double cover of an abelian surface $Z$
whose branch divisor $R\subseteq Z$ has at worst simple singularities,
and $W \to \wt W$ be the resolution of singularities.
If $W$ is of general type, then $R$ is ample and we have the following isomorphism between the fundamental groups:
\begin{equation}\label{eqn-lem-step-III-1}
\zeta_*:~\pi_1(W) \overset{\cong}\lra \pi_1(Z) \cong \mathbb Z^{\oplus 4},\qquad
\text{where $\zeta:\,W \to Z$ is the composition}.
\end{equation}
In particular, any unramified cover $\tau:\,W' \to W$ is induced from a unramified cover $\upsilon:\,Z' \to Z$
with the following commutative diagram
$$\xymatrix{W'\ar[rr]^-{\tau} \ar[d]_-{\zeta'} && W \ar[d]^-{\zeta}\\
Z' \ar[rr]^-{\upsilon} && Z}$$
\end{lemma}
\begin{proof}
First we show that $R$ is ample.
According to \eqref{eqn-step-III-i-2} together with our assumption, we have $0<K_Z^2=\frac12R^2$,
from which it follows that $R^2>0$.
For any curve $C\subseteq Z$, $R\cdot C \geq 0$ since $Z$ is an abelian surface;
otherwise, $C^2<0$ according to the Hodge index theorem,
which is impossible as $C$ is contained in an abelian surface.
This shows that $R\cdot C>0$. Hence $R$ is ample.

Next we prove the isomorphism \eqref{eqn-lem-step-III-1}.
Since $R$ is the branch divisor of the double cover $\tilde\zeta$,
there exists a line bundle $\mathcal L$ such that $\mathcal O_{Z}(R) \equiv \mathcal L^{\otimes 2}$.
Consider the complete linear system $|R|$ as a projective space and let $\Lambda\subseteq |R|$ be the open and dense
subvariety corresponding to reduced divisors with at worst simple singularities.
For each $\lambda\in \Lambda$, let $R_{\lambda}\subseteq |R|$ be the corresponding divisor and
$\tilde\zeta_{\lambda}:\,\wt W_{\lambda} \to Z$ the double cover determined by the relation
$\mathcal O_{Z}(R_{\lambda}) \equiv \mathcal L^{\otimes 2}$.
Then we obtain a flat family of projective surfaces
$$\tilde f:\,\bigcup_{\lambda\in\Lambda} \wt W_{\lambda} \lra \Lambda.$$
Moreover, by our assumption, each fibre of $\tilde f$ is irreducible and reduced with at worst rational double points
as its singularities.
Hence every fibre of $\tilde f$ has the same fundamental group by \cite[Lemma\,9]{xiao-91}.
On the other hand, since $\mathcal O_{Z}(R) \equiv \mathcal L^{\otimes 2}$, it follows that
the complete linear system $|R|$ is base-point-free.
Thus one can choose an element $R_{\lambda}\subseteq |R|$ such that $R_{\lambda}$ is smooth.
Hence according to \cite[Corollary\,2.7]{nori-83}, one has the following isomorphism:
$$(\tilde\zeta_{\lambda})_*:~\pi_1(\wt W_{\lambda}) \lra \pi_1(Z).$$
So we obtain an isomorphism
$$\tilde\zeta_*:~\pi_1(\wt W) \lra \pi_1(Z).$$
Finally, Since $\wt W$ has at worst rational double points as its singularities,
one has a natural isomorphism $\pi_1(W) \to \pi_1(\wt W)$. The proof is complete.
\end{proof}

\begin{proof}[Proof of \autoref{thm-main}]
(1). It suffices to prove that if $K_X^2< \frac92\chi(\mathcal O_X)$, then
\begin{equation}\label{eqn-pf-thm-main-1}
K_X^2\geq 4\chi(\mathcal O_X)+4(q-2).
\end{equation}

We prove \eqref{eqn-pf-thm-main-1} by induction on the degree of the Albanese map of the surface.
According to \autoref{thm-step-II}, $\deg\left(\Alb_{X}\right)>1$.
More precisely, the Albanese map $\Alb_X$ of $X$ factors through a double cover
$\pi:\,X\to Y=X/\langle\sigma\rangle$. Let $Y_0$ be the minimal smooth model of $Y$.
It follows from \autoref{lem-step-III-1} that $Y_0$ is minimal surface of maximal Albanese dimension with
$$q(Y_0)=q,\quad\qquad\deg\left(\Alb_{Y_0}\right)=\frac12\deg\left(\Alb_X\right).$$
Moreover, according to \autoref{prop-step-III-i} and \autoref{lem-step-III-ii-1}(2), we may assume that
$$K_{Y_0}^2 < \frac92\chi(\mathcal O_{Y_0}).$$
In other word, $Y_0$ is a minimal surface of general type and maximal Albanese dimension with $K_{Y_0}^2 < \frac92\chi(\mathcal O_{Y_0})$ and $\deg\left(\Alb_{Y_0}\right)=\frac12\deg\left(\Alb_X\right)$.
By induction, we obtain
$$K_{Y_0}^2 \geq 4\chi(\mathcal O_{Y_0})+4\big(q(Y_0)-2\big).$$
Hence \eqref{eqn-pf-thm-main-1} follows from \autoref{lem-step-III-ii-1}(1).

(2). We use the same notations as above.
First it is clear that the equality \eqref{eqn-main-2} holds if the canonical model of $X$ satisfies the condition in the theorem.
Conversely, if \eqref{eqn-main-2} holds,
then the Albanese map $\Alb_X$ of $X$ factors through a double cover
$\pi$ by \autoref{thm-step-II} since $K_X^2<\frac92\chi(\mathcal O_X)$.
According to the proof of (1) together with \autoref{prop-step-III-i} and \autoref{lem-step-III-ii-1},
our theorem is proved unless $q=2$, $\pi=\pi_0$ is unramified and $Y=X/\langle\sigma\rangle$ is minimal of general type
of maximal dimension with
$$q(Y)=q,\qquad K_Y^2=4\chi(\mathcal O_Y).$$
By repeating the above process, we obtain that if the equality \eqref{eqn-main-2} holds,
then either the canonical model of $X$ is a flat double cover of $Y$ as required in our theorem,
or $q=2$ and the Albanese map $\Alb_X$ factors as
$$\xymatrix{Z_0=X \ar[r]^-{\zeta_1=\pi} \ar@/_5mm/"1,5"_-{\Alb_X}  & Z_1=Y \ar[r]^-{\zeta_2} & \cdots  \ar[r]^-{\zeta_{s-1}}
&Z_{s-1} \ar[r]^-{\zeta_s}& Z_s=\Alb_X(X),}$$
where $s\geq 2$, $\zeta_i$ is a unramified double cover for any $1\leq i\leq s-1$,
$Z_{s-1}$ is of general type, $Z_s$ is an abelian surface,
and $\zeta_s$ is the resolution of a flat double cover $\tilde \zeta_{s-1}:\, \wt Z_{s-1} \to Z_s$
whose branch divisor has at worst simple singularities.
To complete the proof, it suffices to derive a contradiction if the later possibility happens.

Assume that the later case happens. Then $Z_i$ is of maximal Albanese dimension with
\begin{equation}\label{eqn-step-III-2}
\deg\left(\Alb_{Z_i}\right)=2^{s-i},\qquad \forall~0\leq i\leq s.
\end{equation}
By \autoref{lem-step-III-2}, there exists an abelian surface $Z'$ and an unramified double cover $\upsilon:\,Z' \to Z_s$
with the following commutative diagram.
$$\xymatrix{
Z_{s-2}\ar[rr]^-{\zeta_{s-1}} \ar[d]_-{\zeta'} && W \ar[d]^-{\zeta_s}\\
Z' \ar[rr]^-{\upsilon} && Z_s}$$
Since $Z'$ is an abelian surface and $\deg(\zeta')=2$,
it follows that $\deg\left(\Alb_{Z_{s-2}}\right)$ is at most $2$.
It is a contradiction to \eqref{eqn-step-III-2}.
This completes the proof.
\end{proof}

\begin{proof}[Proof of \autoref{cor-main-1}]
It follows directly from \autoref{thm-main} except
the ampleness of the branch divisor in characterizing the equality $K_X^2=4\chi(\mathcal O_X)$,
which comes from \autoref{lem-step-III-2}.
\end{proof}

\enddocument